\documentclass[a4paper,12pt]{amsart}

\usepackage{geometry}
\usepackage{amsthm,amsmath,natbib}
\usepackage{amsfonts, mathtools}
\usepackage[lowtilde]{url}
\RequirePackage[colorlinks,citecolor=blue,urlcolor=blue]{hyperref}
\usepackage{randtext}

\newtheorem*{theorem*}{Theorem}
\newcommand{\D}{D}																		
\newcommand{\PP}{\mathsf P}														
\newcommand{\R}{\mathbb R}														
\newcommand{\tr}{^\mathsf{T}}													
\newcommand{\eqd}{\stackrel{\mathclap{\mbox{\tiny{\emph{d}}}}}{=}}

\author{Stanislav Nagy}

\email{\randomize{nagy@karlin.mff.cuni.cz}}

\address{
	Charles University,
	Department of Probability and Math. Statistics,
	Sokolovsk\'a 83, 186 75  Praha 8,
	Czech Republic
}

\title{Halfspace depth does not characterize probability distributions}
\date{\today}


\begin{document}

\begin{abstract}
We give examples of different multivariate probability distributions whose halfspace depths coincide at all points of the sample space. 
\end{abstract}

\maketitle

For a $d$-variate random vector $X$ with distribution $P$ and $x \in \R^d$, the halfspace (or Tukey) depth of $x$ with respect to $P$ is given by
	\[	\D(x;P) = \inf_{u \in \R^d \setminus \{0\}} \PP\left( u\tr X \leq u\tr x \right).	\]
This function, proposed by \citet{Tukey1975}, constitutes a base for nonparametric statistical analysis of multivariate data. A vast body of literature on the depth, and depth-based statistical procedures exists; see, for instance, \citet{Donoho_Gasko1992}, \citet{Liu_etal1999}, \citet{Rousseeuw_Ruts1999}, \citet{Zuo_Serfling2000}, and \citet{Einmahl_etal2015}.

The depth characterization conjecture is perhaps the most fundamental problem in the theory of data depth. The conjecture states that for any two distinct probability distributions $P, Q$ in $\R^d$ there exists a point $x \in \R^d$ at which the depths $\D(x;P)$ and $\D(x;Q)$ differ. It is easy to see that the conjecture holds true for $d=1$. A positive answer for $d\geq 2$ would justify the pursuit of depth-based inference for multivariate data, as there would be a one-to-one map between all depth surfaces and probability distributions. The depth characterization problem was studied by many authors, including \citet{Struyf_Rousseeuw1999}, \citet{Koshevoy2002, Koshevoy2003}, \citet{Hassairi_Regaieg2007, Hassairi_Regaieg2008}, \citet{Cuesta_Nieto2008c}, and \citet{Kong_Zuo2010}. From these results we know that the conjecture holds true if, for instance, $P$ and $Q$ have finite support, or if their depths satisfy certain smoothness conditions. 

In this note we show that the general conjecture is not true. We do so by constructing two different probability distributions with the same depth.

\begin{theorem*}
For all $d\geq 2$ there exist probability distributions $P, Q$ in $\R^d$ such that $P \ne Q$, and $\D(x;P) = \D(x;Q)$ for all $x \in \R^d$. 
\end{theorem*}

\begin{proof}
For $t = \left(t_1, \dots, t_d\right)\tr \in \R^d$ and $0<\alpha<\infty$ denote $\left\Vert t \right\Vert_{\alpha} = \left( \sum_{i=1}^d \left\vert t_i \right\vert^\alpha \right)^{1/\alpha}$; for $\alpha = \infty$ set $\left\Vert t \right\Vert_\infty = \max_{i=1,\dots,d} \left\vert t_i \right\vert$.

Consider $X = \left(X_1, \dots, X_d\right)\tr \sim P$ and $Y = \left(Y_1, \dots, Y_d\right)\tr \sim Q$ given by their characteristic functions
	\begin{equation*}	
	\begin{aligned}
	\psi_X(t) & = \exp\left(-\left\Vert t \right\Vert_{1}^{1/2} \right) = \exp\left(-\sqrt{ \sum_{j=1}^d \left\vert t_j \right\vert} \right) \\
	\psi_Y(t) & = \exp\left(-\left\Vert t \right\Vert_{1/2}^{1/2} \right) = \exp\left(- \sum_{j=1}^d \sqrt{\left\vert t_j \right\vert} \right)
	\end{aligned} \quad\mbox{for }t = \left(t_1, \dots, t_d \right)\tr \in\R^d.
	\end{equation*}
In what follows we show that the random vectors $X$ and $Y$ are well defined, with different distributions, yet 
	\begin{equation*}	\label{D}
	\D(x;P) = \D(x;Q) = F\left(-\left\Vert x \right\Vert_{\infty} \right)	\quad\mbox{for all }x\in\R^d	
	\end{equation*}
for a fixed function $F$.

Functions $\psi_X$ and $\psi_Y$ are both positive definite. That follows by results of \citet{Levy1937} and \citet{Schoenberg1938}, see \citet[Section~1]{Koldobsky2009} or \citet[Problem~2]{Zastavnyi2000}. Thus, by Bochner's theorem \citep[Theorem~1.8.9]{Ushakov1999}, random vectors $X$ and $Y$ are well defined, and $P \ne Q$. The characteristic function $\psi_X$ depends on $t$ only through $\left\Vert t \right\Vert_\alpha$ with $\alpha = 1$. Thus, $P$ takes the form of a so-called $\alpha$-symmetric distribution \citep[Definition~7.1]{Fang_etal1990}. Likewise, $Q$ is $\alpha$-symmetric with $\alpha = 1/2$. For any $d$-dimensional random vector $Z = \left(Z_1, \dots, Z_d \right)\tr$ with $\alpha$-symmetric distribution, $\alpha \in (0,2]$, it is known that $u\tr Z \eqd \left\Vert u \right\Vert_\alpha Z_1$ for all $u \in \R^d$ \citep[Theorem~7.1]{Fang_etal1990}. Here, $\eqd$ stands for ``is equal in distribution".

The characteristic functions of marginal distributions $X_1$ and $Y_1$ are given by
	\[	
	\psi_X\left((t_1,0, \dots, 0)\tr\right) = \psi_Y\left((t_1,0, \dots, 0)\tr\right) = \exp\left(-\sqrt{\left\vert t_1 \right\vert} \right) \quad\mbox{for }t_1 \in \R.
	\]	
Therefore, since characteristic functions uniquely determine distributions \citep[page~54]{Ushakov1999}, $X_1 \eqd Y_1$. Denote the common distribution function of $X_1$ and $Y_1$ by $F$. 

For $x \in \R^d$ we can write
	\[
	\begin{aligned}
	\D(x;P) & = \inf_{u \in \R^d \setminus \{0\}} \PP\left( u\tr X \leq u\tr x \right) = \inf_{u \in \R^d \setminus \{0\}} \PP\left( \left\Vert u \right\Vert_{\alpha} X_1 \leq u\tr x \right) \\
	& = \inf_{u \in \R^d \setminus \{0\}} F\left( \frac{u\tr x}{\left\Vert u \right\Vert_{\alpha}} \right) = F\left( \inf_{u \in \R^d \setminus \{0\}} \frac{u\tr x}{\left\Vert u \right\Vert_{\alpha}} \right)
	\end{aligned}
	\]
with $\alpha = 1$, and likewise 
	\[	\D(x;Q) = F\left( \inf_{u \in \R^d \setminus \{0\}} \frac{u\tr x}{\left\Vert u \right\Vert_{\alpha}} \right)	\]
with $\alpha = 1/2$. It remains to show that 
	\[	\inf_{u \in \R^d \setminus \{0\}} \frac{u\tr x}{\left\Vert u \right\Vert_{1}} = \inf_{u \in \R^d \setminus \{0\}} \frac{u\tr x}{\left\Vert u \right\Vert_{1/2}} = - \left\Vert x \right\Vert_\infty,	\]
which follows from a version of H\"older's inequality \citep[Lemma~A.1]{Chen_Tyler2004}.
\end{proof}

Note that measures $P$ and $Q$ constructed in the proof above are fairly common in statistics and probability --- they are both multivariate extensions of the L\'evy symmetric $1/2$-stable distribution \citep{Levy1937}. Other probability measures whose depths coincide are immediately available from our construction. These include families of multivariate $\alpha$-stable distributions with any $0< \alpha \leq 1$, or distributions whose characteristic functions depend on the Minkowski functional of certain symmetric star bodies \citep{Koldobsky2009}. An interesting open question that remains to be settled is a complete description of probability distributions characterized by their halfspace depth. 

\subsection*{Acknowledgement}
Supported by the grant 18-00522Y of the Czech Science Foundation, and by the PRIMUS/17/SCI/3 project of Charles University.

\def\cprime{$'$} \def\polhk#1{\setbox0=\hbox{#1}{\ooalign{\hidewidth
  \lower1.5ex\hbox{`}\hidewidth\crcr\unhbox0}}}



\end{document}